\documentclass[11pt]{amsart}
\usepackage{amssymb}

\title[sextic variety]{Sextic variety as Galois closure variety of smooth cubic}
\author[Hisao Yoshihara]{}

\newtheorem{theorem}{Theorem}[section]
\newtheorem{lemma}[theorem]{Lemma}
\newtheorem{proposition}[theorem]{Proposition}

\newtheorem{corollary}[theorem]{Corollary}
\newtheorem{claim}[theorem]{Claim}
\newtheorem{rep}{Representation}

\theoremstyle{definition}
\newtheorem{definition}[theorem]{Definition}
\newtheorem{example}[theorem]{Example}
\theoremstyle{remark}
\newtheorem{remark}[theorem]{Remark}

\newenvironment{namelist}[1]{%
\begin{list}{}
  {
   \settowidth{\labelwidth}{#1}
   \setlength{\leftmargin}{2.5\labelwidth}}
}{%
\end{list}}

\begin{document}
\maketitle

\begin{center}

{\sc Hisao Yoshihara}\\
\medskip
{\small{\em Department of Mathematics, Faculty of Science, Niigata University,\\
Niigata 950-2181, Japan}\\
E-mail:{\tt yosihara@math.sc.niigata-u.ac.jp}}
\end{center}

\bigskip

\begin{abstract}
Let $V$ be a nonsingular projective algebraic variety of dimension $n$. Suppose there exists a very ample divisor $D$ such that 
$D^n=6$ and $\dim{\mathop{\mathrm{H^0}}\nolimits}(V,\ {\mathcal O}(D))=n+3$. Then, $(V, D)$ defines a $D_6$-Galois embedding if 
and only if it is a Galois closure variety of a smooth cubic in $\mathbb P^{n+1}$ with respect to a suitable projection center 
such that the pull back of hyperplane of $\mathbb P^n$ is linearly equivalent to $D$. 
\end{abstract}

\bigskip

\section{Introduction}
The purpose of this article is to generalize the following assertion (cf. \cite[Theorem 4.5]{y5}) to $n$-dimensional varieties.  

\begin{proposition}\label{27}
Let $C$ be a smooth sextic curve in ${\mathbb P}^3$ and assume the genus is four. 
If $C$ has a Galois line, then the group $G$ is isomorphic to the cyclic or dihedral group of order six. 
Moreover, $G$ is isomorphic to the latter one if and only if $C$ is obtained as the Galois closure  curve of a smooth plane cubic $E$ 
with respect to a point $P \in \mathbb P^2 \setminus E$, where $P$ does not lie on the tangent line to $E$ at any flex.  
\end{proposition}

Before going into the details, we recall the definition of Galois embeddings of algebraic varieties and the relevant results. 
In this article a variety, a surface and a curve will mean a nonsingular projective algebraic variety, surface and curve, 
respectively. 

Let $k$ be the ground field of our discussion, we assume it to be an algebraically closed field of characteristic zero.  
Let $V$ be a variety of dimension $n$ with a very ample divisor $D$; we 
denote this by a pair $(V, D)$. Let $f=f_D:V \hookrightarrow {\mathbb P}^N$ be the embedding of $V$ associated
 with the complete linear system $|D|$, where $N+1=\dim{\mathop{\mathrm{H^0}}\nolimits}(V,\ {\mathcal O}(D))$. Suppose $W$ is a linear
 subvariety of ${\mathbb P}^N$ satisfying $\dim W=N-n-1$ and $W \cap f(V)=\emptyset$. Consider the projection $\pi_W$ 
from $W$ to $\mathbb P^n$, i.~ e., $\pi_W : {\mathbb P}^N \dashrightarrow \mathbb P^n$. Restricting $\pi_W$ onto $f(V)$, we get a surjective morphism $\pi=\pi_W \cdot f : V \longrightarrow \mathbb P^n$. 

Let $K=k(V)$ and $K_0=k(\mathbb P^n)$ be the function fields of $V$ and $\mathbb P^n$ respectively. The covering map $\pi$ 
induces a finite extension of fields ${\pi}^* : K_0 \hookrightarrow K$ of degree $\deg f(V)=D^n$, which is the
 self-intersection number of $D$. 
We denote by $K_W$ the Galois closure of this extension and by $G_W=Gal(K_W/K_0)$ the Galois group of $K_W/K_0$. 
By \cite{ha} $G_W$ is isomorphic to the monodromy group of the covering $\pi : V \longrightarrow \mathbb P^n$.
Let $V_W$ be the $K_W$-normalization of $V$ (cf. \cite[Ch.2]{ii}). 
Note that $V_W$ is determined uniquely by $V$ and $W$. 

\begin{definition}\label{1}
In the above situation we call $G_W$ and $V_W$ the Galois group and the Galois closure variety at $W$ 
respectively (cf. \cite{y6}). 
If the extension $K/K_0$ is Galois, then we call $f$ and $W$ a Galois embedding and a Galois subspace for
 the embedding respectively. 
\end{definition}

\begin{definition}\label{3}
A variety $V$ is said to have a Galois embedding if there exist a very ample 
divisor $D$ satisfying that the embedding associated with $|D|$ has a Galois subspace. In this case the pair $(V,D)$ 
is said to define a Galois embedding.
\end{definition}

If $W$ is a Galois subspace and $T$ is a projective transformation of ${\mathbb P}^N$, then $T(W)$ is a Galois
 subspace of the embedding $T \cdot f$. Therefore the existence of Galois subspace does not depend on the choice 
of the basis giving the embedding. 

\begin{remark}\label{31}
If a variety $V$ exists in a projective space, then by taking a linear subvariety, we can define a Galois subspace and Galois group 
similarly as above. 
Suppose $V$ is not normally embedded and there exists a linear subvariety $W$ such that the projection 
$\pi_W$ induces a Galois extension of fields. Then, taking $D$ as a hyperplane section of $V$ in the embedding,
 we infer readily that $(V,D)$ defines a Galois embedding with the same Galois group in the above sense.
\end{remark}

By this remark, for the study of Galois subspaces, it is sufficient to consider the case where $V$ is normally embedded. 

\medskip

We have studied Galois subspaces and Galois groups for hypersurfaces in \cite{y1}, \cite{y2} and \cite{y3} and space 
curves in \cite{y5} and \cite{y7}. 
The method introduced in \cite{y6} is a generalization of the ones used in these studies. 

\medskip

Hereafter we use the following notation and convention:

\begin{namelist}{}
\item[$\cdot$]${\rm Aut}(V)$ : the automorphism group of a variety $V$
\item[$\cdot$]$\langle a_1, \cdots, a_m \rangle$ : the subgroup generated by $a_1, \cdots, a_m$ 
\item[$\cdot$]$D_{2m}$ : the dihedral group of order $2m$
\item[$\cdot$]$|G|$ : the order of a group $G$
\item[$\cdot$]$\sim$ : the linear equivalence of divisors
\item[$\cdot$]${\bf 1}_m$ : the unit matrix of size $m$
\item[$\cdot$]$X \ast Y$ : the intersection cycle of cycles $X$ and $Y$ in a variety.
\item[$\cdot$]$(X_0: \cdots : X_m)$ : a set of homogeneous coordinates on ${\mathbb P}^m$
\item[$\cdot$]$g(C)$ : the genus of a smooth curve $C$
\item[$\cdot$] For a mapping $\varphi : X \longrightarrow Y$ and a subset $X' \subset X$, we often use the same $\varphi$ to denote the restriction $\varphi|_{X'}$. 
\end{namelist}

\bigskip

\section{Results on Galois embeddings}

We state several properties concerning Galois embedding without the proofs, for the details, see \cite{y6}. 
By definition, if $W$ is a Galois subspace, then each element $\sigma$ of $G_W$ is an automorphism of $K=K_W$ over $K_0$. Therefore it induces a birational transformation of $V$ over $\mathbb P^n$. This implies that $G_W$ can be viewed as a subgroup of ${\rm Bir}(V/\mathbb P^n)$, the group of birational transformations of $V$ over $\mathbb P^n$. Further we can say the following:

\begin{rep}\label{a1}
Each birational transformation belonging to $G_W$ turns out to be regular on $V$, hence we have a faithful representation 
$$
\alpha :  G_W \hookrightarrow {\rm Aut}(V). \eqno(1)
$$
\end{rep}

\medskip

\begin{remark}
Representation 1 is proved by using transcendental method in \cite{y6}, however we can prove it algebraically by making use of the 
results \cite[Ch. I, 5.3. Theorem 7]{sh} and \cite[Ch. V, Theorem 5.2]{har}. 
\end{remark}

\medskip

Therefore, if the order of ${\rm Aut}(V)$ is smaller than the degree $d$, then $(V,D)$ cannot define a Galois embedding. 
In particular, if ${\rm Aut}(V)$ is trivial, then $V$ has no Galois embedding. 
On the other hand, in case $V$ has an infinitely many automorphisms, we have examples such that there exist infinitely many distinct Galois embeddings, see Example 4.1 in \cite{y6}. 
\medskip

When $(V,D)$ defines a Galois embedding, we identify $f(V)$ with $V$. Let $H$ be a hyperplane of ${\mathbb P}^N$ containing $W$ and put 
$D'=V \ast H$. 
Since $D' \sim D$ and ${\sigma}^*(D')=D'$, for any $\sigma \in G_W$, we see $\sigma$ induces an automorphism of ${\mathop{\mathrm{H^0}}\nolimits}(V, {\mathcal O}(D))$. 
This implies the following. 

\begin{rep}\label{18}
We have a second faithful representation 

$$\beta : G_W \hookrightarrow PGL(N+1, k).\eqno(2)$$

\end{rep}

\medskip
In the case where $W$ is a Galois subspace we identify $\sigma \in G_W$ with $\beta(\sigma) \in PGL(N+1, k)$ hereafter. 
Since $G_W$ is a finite subgroup of ${\rm Aut}(V)$, we can consider the quotient $V/G_W$ and let $\pi_G$ be the quotient morphism, $\pi_G : V \longrightarrow V/G_W$.

\begin{proposition}\label{15}
If $(V,D)$ defines a Galois embedding with the Galois subspace $W$ such that the projection is $\pi_W : {\mathbb P}^N \dashrightarrow \mathbb P^n$, then there exists an isomorphism $g:V/G_W \longrightarrow \mathbb P^n$ satisfying $g \cdot \pi_G = \pi$. Hence the projection $\pi$ is a finite morphism and the fixed loci of $G_W$ consist of only divisors.
\end{proposition}

Therefore, $\pi$ turns out to be a Galois covering in the sense of Namba \cite{na}.

\begin{lemma}\label{49}
Let $(V,D)$ be the pair in Proposition \ref{15}. 
Suppose $\tau \in G$ has the representation 
\[
\beta(\tau) =[1, \ldots, 1, e_m], \ (m \geq 2)
\]
where $e_m$ is an $m$-th root of unity. Let $p$ be the projection from $(0: \cdots :0:1) \in W$ to $\mathbb P^{N-1}$. 
Then, $V/\langle \tau \rangle$ is isomorphic to $p(V)$ if $p(V)$ is a normal variety. 
\end{lemma}

We have a criterion that $(V,D)$ defines a Galois embedding. 

\begin{theorem}\label{8}
The pair $(V, D)$ defines a Galois embedding if and only if the following conditions hold{\rm :}
\begin{namelist}{3}
\item[{\rm (1)}]There exists a subgroup $G$ of ${\rm Aut}(V)$ satisfying that $|G|=D^n$.
\item[{\rm (2)}]There exists a $G$-invariant linear subspace ${\mathcal L}$ of ${\mathop{\mathrm{H^0}}\nolimits}(V, {\mathcal O}(D))$ of dimension $n+1$ such that, for any $\sigma \in G$, the restriction ${\sigma}^*|_{\mathcal L}$ is a multiple of the identity. 
\item[{\rm (3)}]The linear system ${\mathcal L}$ has no base points. 
\end{namelist}
\end{theorem}

It is easy to see that $\sigma \in G_W$ induces an automorphism of $W$, hence we obtain another representation of $G_W$ as follows. 
Take a basis $\{f_0, f_1, \ldots, f_N  \}$ of ${\mathop{\mathrm{H^0}}\nolimits}(V, {\mathcal O}(D))$ satisfying that $\{f_0, f_1, \ldots, f_n  \}$ is a basis of ${\mathcal L}$ in Theorem \ref{8}. Then we have the representation  

\[
\newcommand{\bg}{%
 \family{cmr}\size{20}{12pt}\selectfont}
 \newcommand{\bigzerol}{\smash{\hbox{\bg 0}}}
 \newcommand{\bigzerou}{%
  \smash{\lower1.7ex\hbox{\bg 0}}}
\beta_1(\sigma)=\begin{pmatrix}\lambda_{\sigma} & & & \vdots &  \\
 & \ddots & & \vdots & {\bf *} \\
 & & \lambda_{\sigma} & \vdots & \\
 \cdots & \cdots & \cdots & \vdots & \cdots  \\
& {\bf 0} & & \vdots & M_{\sigma}   
\end{pmatrix}. \eqno(3) 
\]
Since the projective representation is completely reducible, we get another 
representation using a direct sum decomposition:
\[
\beta_2(\sigma)=\lambda_{\sigma} \cdot {\bf 1}_{n+1}\oplus M'_{\sigma}. 
\]
Thus we can define
\[
\gamma(\sigma)=M'_{\sigma} \in PGL(N-n, k). 
\]
Therefore $\sigma$ induces an automorphism on $W$ given by  $M'_{\sigma}$. 

\begin{rep}\label{19}
We get a third representation 

$$
\gamma : G_W \longrightarrow PGL(N-n, k).\eqno(4)
$$

\end{rep}

\bigskip

Let $G_1$ and $G_2$ be the kernel and image of $\gamma$ respectively.

\begin{theorem}\label{a3}
We have an exact sequence of groups 
\[
1 \longrightarrow G_1 \longrightarrow G \stackrel{\gamma}{\longrightarrow} G_2 \longrightarrow 1, 
\]
where $G_1$ is a cyclic group.
\end{theorem}

\begin{corollary}\label{a4}
If $N=n+1$, i.e., $f(V)$ is a hypersurface, then $G$ is a cyclic group.
\end{corollary}

This assertion has been obtained in \cite{y3}. 
Moreover we have another representation.

Suppose that $(V, D)$ defines a Galois embedding and let $G$ be a Galois group at some Galois subspace $W$. Then, 
take a general hyperplane $W_1$ of $\mathbb P^n$ and put $V_1={\pi}^*(W_1)$. The divisor $V_1$ has the 
following properties{\rm :}
\begin{namelist}{(iii)}
\item[\rm{(i)}]If $n \geq 2$, then $V_1$ is a smooth irreducible variety. 
\item[\rm{(ii)}]$V_1 \sim D$. 
\item[\rm{(iii)}]${\sigma}^*(V_1)=V_1$ for any  $\sigma \in G$. 
\item[\rm{(iv)}] $V_1/G$ is isomorphic to $W_1$.
\end{namelist}
Put $D_1=V_1 \cap H_1$, where $H_1$ is a general hyperplane of ${\mathbb P}^N$. Then $(V_1, D_1)$ 
defines a Galois embedding with the Galois group $G$ (cf. Remark \ref{31}). 
Iterating the above procedures, we get a sequence of pairs $(V_i, D_i)$ such that 
\[
(V,D) \supset (V_1, D_1) \supset \cdots \supset (V_{n-1}, D_{n-1}). \eqno(5)
\]
These pairs satisfy the following properties:
\begin{enumerate}
\item[(a)]$V_i$ is a smooth subvariety of $V_{i-1}$, which is a hyperplane section of $V_{i-1}$, 
where $D_i=V_{i+1}$, $V=V_0$ and $D=V_1$ ($1 \le i \le n-1$). 
\item[(b)] $(V_i, D_i)$ defines a Galois embedding with the same Galois group $G$.
\end{enumerate} 

\begin{definition}\label{40}
The above procedure to get the sequence $(5)$ is called the Descending Procedure. 
\end{definition}

Letting $C$ be the curve $V_{n-1}$, we get the next fourth representation. 

\bigskip

\begin{rep}\label{16}
We have a fourth faithful representation 
$$
\delta : G_W \hookrightarrow {\rm Aut}(C), \eqno(6)
$$
where $C$ is a curve in $V$ given by $V \cap L$ such that $L$ is a general linear subvariety of
 ${\mathbb P}^N$ with dimension $N-n+1$ containing $W$. 
\end{rep}

Since the Inverse Problem of Galois Theory over $k(x)$ is affirmative (\cite{mm0}), 
we can prove the following. 

\begin{remark}
Giving any finite group $G$, there exists a smooth curve and very ample divisor $D$ 
such that $(C, D)$ defines a Galois embedding with the Galois group $G$. 
\end{remark}

\section{Statement of results}
Let $V$ be a variety of dimension $n$. 
We say that $V$ has the property $(\P_n)$ if 
\begin{enumerate}
\item there exists a very ample divisor $D$ with $D^n=6$, and 
\item $\dim{\mathop{\mathrm{H^0}}\nolimits}(V,\ {\mathcal O}(D))=n+3$. 
\end{enumerate}

An example of such a variety is a smooth $(2,3)$-complete intersection, where $D$ is a hyperplane section.  
In particular, in case $n=1$, $V$ is a non-hyperelliptic curve of genus four and $D$ is a canonical divisor.  
In case $n=2$, $V$ is a $K3$ surface such that there exists a very ample divisor $D$ with $D^2=6$. 
However, the variety with the property $(\P_n)$ is not necessarily the complete intersection, see Remark \ref{47} below.

We will study the Galois embedding of $V$ for the variety with the property $(\P_n)$. 
Clearly the Galois group is isomorphic to the cyclic group of order six or $D_6$. In the latter case we say that $(V, D)$ defines a $D_6$-embedding or, more simply $V$ 
has a $D_6$-embedding. 

\begin{theorem}\label{100}
Assume $V$ has the property $(\P_n)$. If $V$ has a $D_6$-embedding, then $V$ is obtained as the Galois closure 
variety of a smooth cubic $\Delta$ in ${\mathbb P}^{n+1}$ with respect to a suitable projection center. 
\end{theorem}

Next we consider the converse assertion. Let $\Delta$ be a smooth cubic of dimension $n$ in $\mathbb P^{n+1}$. 
Take a non-Galois point $P \in \mathbb P^{n+1} \setminus \Delta$. 
Note that, for a smooth hypersurface $X \subset \mathbb P^{n+1}$, the number of Gaois points is 
at most $n+2$. The maximal number is attained if and only if $X$ is projectively equivalent to the Fermat variety (cf. \cite{y3}). 

Define the set $\Sigma_P$ of lines as  

\medskip

$\Sigma_P=\{ \ \ell \ | \ \ell \ \mathrm{is \ a \ line \ passing \ through} \ P \ \mathrm{such \ that} \ \ell \ast \Delta \ \mathrm{can \ be \ expressed \ as} \ 
2P_1+P_2, \ \mathrm{where} \ P_i \in \Delta \ (i=1, 2) \ \ \mathrm{and} \ P_1 \ne P_2 \} $

\medskip

The closure of the set $\bigcup_{\ell \in \Sigma_P} \ell$ is a cone, we denote it by $C_P(\Delta)$. 
Then we have the following. 

\begin{lemma}\label{42}
The cone $C_P(\Delta)$ is a hypersurface of degree six. 
\end{lemma}

We can express as $\Delta \ast C_P(\Delta)=2R_1+R_2$, where $R_1$ and $R_2$ are different divisors on $\Delta$. 

\begin{definition}\label{43}
We call $P$ a good point if  
\begin{enumerate}
\item $R_2$ is smooth and irreducible in case $n \geq 2$, or 
\item $R_2$ consists of six points in case $n=1$.
\end{enumerate}

\end{definition}

\begin{proposition}\label{44}
If $P$ is a general point for $\Delta$, then $P$ is a good point.
\end{proposition}

To some extent the converse assertion of Theorem \ref{100} holds as follows. 

\begin{theorem}\label{101}
If $\Delta_P$ is a Galois closure variety of a smooth cubic $\Delta \subset \mathbb P^{n+1}$, 
where the projection center $P$ is a good point, then $\Delta_P$ is a smooth $(2,3)$-complete intersection in $\mathbb P^{n+2}$ with $D_6$-embedding.  
\end{theorem}

\begin{remark}\label{45}
In the assertion of Theorem \ref{101}, the construction of the Galois closure is closely related to the one in \cite[Tokunaga]{to}. In case $n=2$, the Galois closure surface is a $K3$ surface. 
\end{remark}

Applying the Descending Procedure to the variety of Theorem \ref{101}, we get the following. 

\begin{proposition}\label{46}
If a variety $V$ is a smooth $(2,3)$-complete intersection and has a $D_6$-embedding, then there exists the following sequence of varieties $V_i$, 
where $V_i$ has the same properties as $V$ does, i.e., 
\begin{enumerate}
\item[(i)] $V_i$ is a subvariety of $V_{i-1} \ (i \geq 1)$, where $V_0=V$.  
\item[(ii)] $V_i$ is also a smooth $(2,3)$-complete intersection of hypersurfaces in $\mathbb P^{n+2-i}, \ 0 \le i \le n-1$, 
\item[(iii)] $V_i$ has the property $(\P_{n-i})$, 
\item[(iv)] $V_i$ has a $D_6$-embedding. 
\end{enumerate}
The situation above is illustrated as follows: 
\[
\begin{array}{ccccccccc}
\mathbb P^{n+2} & \dashrightarrow & \mathbb P^{n+1} & \dashrightarrow & & \dashrightarrow & \mathbb P^4 & \dashrightarrow & \mathbb P^3 \\
\cup & & \cup & & & & \cup & & \cup \\
V & \supset & V_1 & \supset & \cdots & \supset & V_{n-2} & \supset & V_{n-1} \\
\downarrow & & \downarrow & & & & \downarrow & & \downarrow \\
\mathbb P^{n} & \dashrightarrow & \mathbb P^{n-1} & \dashrightarrow & & \dashrightarrow & \mathbb P^2 & \dashrightarrow & \mathbb P^1,  
\end{array}
\]
where $\dashrightarrow$ is a point projection, $\downarrow $ is a triple covering, $V_{n-2}$ and $V_{n-1}$ are a $K3$ surface and a sextic curve, respectively.  
\end{proposition}

Here we present examples. 

\begin{example}\label{29}
Let $\Delta$ be the smooth cubic in $\mathbb P^3$ defined by 
\[
F(X_0, X_1, X_2, X_3)=X_0^3+X_1^3+X_2^3+X_0^2X_3+X_1X_3^2+X_3^3. \eqno(7)
\]
Let $\pi_P$ be the projection from $P=(0:0:0:1)$ to the hyperplane $\mathbb P^2$. 
Taking the affine coordinates $x=X_0/X_3, \ y=X_1/X_3$ and $z=X_2/X_3$, we get the defining equation of the affine part  
\[
f(x,y,z)=x^3+y^3+z^3+x^2+y+1. 
\]
Put $x=at,\ y=bt$ and $z=ct$. Computing the discriminant $D(f)$ of 
$f(at,bt,ct)=(a^3+b^3+c^3)t^3+a^2t^2+bt+1$ with respect to $t$, 
we obtain 
\[
\begin{array}{ccl}
D(f)&=&-(31a^6-18a^5b-a^4b^2+58a^3b^3-18a^2b^4\\
& & +31b^6+54a^3c^3-18a^2bc^3+58b^3c^3+27c^6). 
\end{array} \eqno(8)
\]
This yields the branch divisor of $\pi_P : \Delta \longrightarrow \mathbb P^2$. 
From $(7)$ and $(8)$ we infer that the defining equation of $2R_1$ is  
\[
F(X_0, X_1, X_2, X_3)=0 \ \mathrm{and} \ (X_0^2+2X_1X_3+3X_3^2)^2=0, 
\]
and that of $R_2$ is
\[
F(X_0, X_1, X_2, X_3)=0 \ \mathrm{and} \ 4X_0^2-X_1^2+2X_1X_3+3X_3^2=0.  
\]
It is not difficult to check that $R_2$ is smooth and irreducible, hence $P$ is a good point for $\Delta$.  
By taking a double covering along this curve \cite{to}, we get the $K3$ surface $\Delta_P$ in ${\mathbb P}^4$ defined by 
$F=0$ and $X_4^2=4X_0^2-X_1^2+2X_1X_3+3X_3^2$, which is a $(2,3)$-complete intersection. The Galois line is given by $X_0=X_1=X_2=0$. 
\end{example}

How is the Galois closure variety when the projection center is not a good point? 
Let us examine the following example. 

\begin{example}\label{30}
For a projection with some center $P \in \mathbb P^3 \setminus \Delta$, the Galois closure surface $\Delta_P$ 
is not necessarily a $K3$ surface.
Indeed, let $\Delta$ be the smooth cubic defined by 
\[
F(X_0, X_1, X_2, X_3)=X_0^3+X_1^3+X_2^3+X_0X_3^2-X_3^3. \eqno(9)
\]
Clearly the point $P=(0:0:0:1)$ is not a Galois one. 
Taking the same affine coordinates as in Example \ref{29}, we get the defining equation of the affine part    
\[
f(x,y,z)=x^3+y^3+z^3+x-1. 
\]
Put $x=at,\ y=bt$ and $z=ct$. Computing the discriminant $D(f)$ of 
$f(at,bt,ct)=(a^3+b^3+c^3)t^3+at-1$ with respect to $t$, 
we obtain 
\[
D(f)=-(31a^3+27b^3+27c^3)(a^3+b^3+c^3).  
 \eqno(10)
\]
This yields the branch divisor of $\pi_P : \Delta \longrightarrow \mathbb P^2$. 
From $(9)$ and $(10)$ we infer that the defining equation of $2R_1$ is  
 $C_1+C_2$, where $C_1$ (resp. $C_2$) is given by 
$X_0^3+X_1^3+X_2^3=0$ (resp. $31X_0^3+27X_1^3+27X_2^3=0$). 
Hence the defining equation of the sextic $C_P(V)$ is 
\[
(X_0^3+X_1^3+X_2^3)(31X_0^3+27X_1^3+27X_2^3)=0. 
\] 
Let $\Delta_P$ be the double covering of $\Delta$ branched along the 
divisor $R_2$, where $R_2=R_{21}+R_{22}$ such that 
$R_{21}$ (resp. $R_{22}$) is given by the intersection of $F=0$ and $X_0-X_3=0$ (resp. $F=0$ and $X_0-3X_3=0$). 
The $R_{2i}$ ($i=1,2$) is a smooth curve on $\Delta$ satisfying that $R_{2i}^2=3$, $R_2^2=12$ and 
$R_{21}.R_{22}=3$. 
We infer that $\Delta_P$ is a normal surface, therefore it is a Galois closure surface at $P$ (Definition \ref{1}). 
However, it has three singular points of type $A_1$, so that it is not a $K3$ surface.  
The minimal resolution of $\Delta_P$ turns out to be a $K3$ surface. 
\end{example}

\begin{remark}\label{47}
The variety with the property $(\P_n)$ is not necessarily a $(2,3)$-complete intersection. For example, in case $n=1$, 
Take $V=C$ as the Galois closure curve of a smooth cubic $\Delta \subset \mathbb P^2$ obtained as follows: 
let $T$ be a tangent line to $\Delta$ at a flex.  Choose a point $P \in T$ satisfying the following condition:
if $\ell_P$ is a line passing through $P$ and $\ell_P \ne T$, then $\ell_P$ does not tangent to $\Delta$ at any flex. 
Let $C$ be the Galois closure curve for the point projection $\pi_P : \Delta \longrightarrow P^1$, i.e., $\widetilde{\pi} : C \longrightarrow \Delta$ is a double covering, which has four branch points (see, for example \cite[pp. 287--288]{my}), hence $g(C)=3$. 
Let $D$ be the divisor ${\widetilde{\pi}}^*(\ell \ast \Delta)$, where $\ell$ is a line passing through $P$. Clearly we have $\deg D=6$, the complete linear system $|D|$ has no base point 
and $\dim {\mathop{\mathrm{H^0}}\nolimits}(C, {\mathcal O}(D))=4$. 
Let $f : C \longrightarrow C'$ be the morphism associated with $|D|$. 
The double covering $\widetilde{\pi}$ factors as $\widetilde{\pi}={\widetilde{\pi}}' \cdot f$, where ${\widetilde{\pi}}' : C' \longrightarrow \Delta$ is a restriction of the projection 
$\mathbb P^3 \dashrightarrow \mathbb P^2$. 
Since $g(C') \geq 1$, we see $\deg C' \ne 2$ and $3$. Hence $\deg C'=6$ and $f$ is a birational morphism. Further, we have the 
projection ${\widetilde{\pi}}' : C' \longrightarrow \Delta$ and $\Delta$ is nonsingular, hence $C'$ is smooth. Therefore $f$ is an isomorphism. 
Since $g(C)=3$, $C$ is not a $(2,3)$-complete intersection. 
\end{remark}

\section{proof}
First we prove Theorem \ref{100}. 
The case $n=1$ have been proved (\cite{y5}). So that we will restrict ourselves to the case $n \geq 2$. 

Since $V$ is embedded into $\mathbb P^{n+2}$ associated with $|D|$, where $D$ is a very ample divisor with $D^n=6$, we can apply the results in Section 2. 
By assumption $V$ has a Galois line $\ell$ such that the Galois group $G=G_{\ell}$ is isomorphic to $D_6$. 
We can assume $G=\langle \sigma, \tau \rangle$ where $\sigma^3=\tau^2=1$ and $\tau \sigma \tau =\sigma^{-1}$.  
Let $\rho_1 : V \longrightarrow V^{\tau}=V/\langle \tau \rangle$. We see $\rho_2 : V^{\tau} \longrightarrow V^{\tau}/G \cong  \mathbb P^n$ 
turns out a morphism. 
Then, we have $\pi=\rho_2 \rho_1 : V \longrightarrow V/G \cong \mathbb P^{n}$. 
Note that $\rho_2$ is a non-Galois triple covering. 
By taking suitable coordinates, we can assume $\ell$ is given by $X_0=X_1= \cdots = X_n=0$. 
As we see in Section 2, we have the representation $\beta : G \hookrightarrow PGL(n+3,k)$. Since the characteristic of $k$ is zero,
 the projective representation is completely reducible, hence  
$\beta(\sigma)$ and $\beta(\tau)$ can be represented as 
\[
{\bf 1}_{n+1} \oplus M_2(\sigma) \ \ \mathrm{and} \ \ {\bf 1}_{n+1} \oplus M_2(\tau),   
\]
respectively, where $M_2(\sigma)$ and $M_2(\tau)$ are in $GL(2,k)$. 
Since $G \cong D_6$, we have the representation   
\newfont{\bg}{cmr10 scaled\magstep4}
\newcommand{\bigzerol}{\smash{\hbox{\bg 0}}}
\newcommand{\bigzerou}{\smash{\lower1.7ex\hbox{\bg 0}}}
\[\beta(\sigma)= \left(
\begin{array}{ccccc}
1 & & & & \bigzerou \\
 & \ddots & & & \\
 & & 1 & & \\ 
 & & & -\frac{1}{2} & \omega+\frac{1}{2} \\
 \bigzerol& & & \omega+\frac{1}{2} & -\frac{1}{2}
\end{array}
\right) \ \ \mathrm{and} \ \ 
\beta(\tau) = 
\left( 
\begin{array}{ccccc}
1 & & & & \bigzerou \\
 & \ddots & & & \\
 & & \ddots & & \\ 
 & & & 1 &  \\
\bigzerol & & &  & -1
\end{array}
\right), 
\]
where $\omega$ is a primitive cubic root of $1$. 
Therefore, the fixed locus of $\tau$ is given by $f(V) \cap H$, where $H$ is the hyperplane defined by $X_{n+2}=0$. 
Put $Z=f(V) \cap H$, i.e., $Z \sim D$. Since $Z$ is ample, it is connected. 
Looking at the representation $\beta(\tau)$, we see $Z$ is smooth, hence it is a smooth irreducible variety.    
Take the point $P=(0: \cdots :0:1) \in \ell$ and an arbitrary point $Q$ in $V$. 
Let $\ell_{PQ}$ be the line passing through $P$ and $Q$. 
Then we have $\tau(\ell_{PQ})=\ell_{PQ}$ and $\tau(V)=V$. 
Let $\pi_P$ be the projection from the point $P$ to the hyperplane $H$. Since $Z$ is smooth, $\pi_P(V)$ is smooth. 
 By Lemma \ref{49}, $\pi_P(V)$ is isomorphic to $V/\langle \tau \rangle$ and we may assume $\pi_P = \rho_1$.   
Therefore we see $V$ is contained in the cone consisting of the lines passing through $P$ and the points in 
$V$. 
Since $\deg V = 6$ and $\deg p = 2$, we conclude the variety $V^{\tau}$ is a smooth cubic in $\mathbb P^{n+1}$. 
This proves Theorem \ref{100}. 

\medskip

Next we prove Lemma \ref{42}.   
Let $H_2$ be a linear variety of dimension two and passing through $P$. 
If $H_2$ is general, then $\Delta \cap H_2$ is a smooth cubic in the plane $H_2 \cong \mathbb P^2$. 
Thus $C_P(\Delta) \cap H_2$ consists of six lines, hence we have $\deg C_P(\Delta)=6$. 

\medskip

The proof of Proposition \ref{44} is as follows. 
Suppose $P$ is a general point for $\Delta$ and let $\pi_P$ be the projection from $P$ to 
the hyperplane $\mathbb P^n$. 
Put $B=\pi_P(R_2)$. 

\begin{claim}\label{102}
The divisor $B$ is irreducible. 
\end{claim}

\begin{proof}
It is sufficient to check in a general affine part. 
Put $x_i=X_i/X_0 \ (i=1, \ldots, n+1)$ and let $f(x_1, \ldots, x_{n+1})$ be the defining equation of an affine part 
$X_0 \ne 0$ of $\Delta$ 
and $P=(u_1, \ldots, u_{n+1}) \in \mathbb A^{n+1}$. Put 
\[
g(u_1, \ldots, u_{n+1},t_0, \ldots, t_n, x)=f(u_1+xt_0, \ldots,  u_{n+1}+xt_n),  
\]
where $(t_0, \ldots, t_n) \in \mathbb P^n$. 
Let $D(g)=D(u_1, \ldots, u_{n+1}, t_0, \ldots, t_{n})$ be the discriminant of $g$ with respect to $x$. 
Owing to \cite[Lemma 3]{y1} and \cite[Claim 1]{y3}, we see $D(g)$ is reduced and irreducible. 
Therefore for a general value $u_1=a_1, \ldots, u_{n+1}=a_{n+1}$, $D(a_1, \ldots, a_{n+1}, t_0, \ldots, t_n)$ 
is irreducible. This implies $B$ is irreducible.  
\end{proof}

\begin{claim}\label{104}
The divisor $R_2$ is irreducible and smooth. 
\end{claim}

\begin{proof}
Suppose $R_2$ is decomposed into irreducible components $R_{21}+ \cdots +R_{2r}$. 
Since $B$ is irreducible, we have $\pi_P(R_{2i})=B$ for each $1 \le i \le r$. 
However, since $\Delta \ast \ell$ has an expression $2P_1+P_2$, the $r$ must be $1$. 
Thus $R_2$ is irreducible. 
Since $\Delta \ast \ell$ has an expression $2P_1+P_2$, where $P_i \in \Delta \ (i=1,2)$, 
$\Delta$ and $\ell$ has a normal crossing at $P_2$ if $P_1 \ne P_2$. In case $P_1=P_2$,  
the intersection number of $\Delta$ and $\ell$ at $P_1$ is three. Since $R_1 \ni P_1 $ and $\Delta \ast C_P(\Delta)=2R_1+R_2$, we see that 
$R_2$ is smooth at $P_1$. 
\end{proof}

This completes the proof of Proposition \ref{44}. 
The proof of Theorem \ref{101} is as follows.

First note that $P$ is not a Galois point. So we consider the Galois closure variety. 
The ramification divisor of $\pi_P : \Delta \longrightarrow \mathbb P^n $ is $2R_1+R_2$. 
The divisor $R_2$ is smooth and irreducible by assumption. 
Let $\Phi$ be the equation of the branch divisor of $\pi_P$. 
As we see in Example \ref{29} ($a=X_0/tX_3,\ b=X_1/tX_3,\ c=X_2/tX_3$) the discriminant is given by the homogeneous equation of $X_0, \ldots, X_n$, 
hence we infer that $\pi_P^*(\Phi)$ has the expression as $\Phi_1^2 \cdot \Phi_2$, where $\Phi_2=0$ defines $R_2$. 
Since $\deg \Phi_2=2$, we can define the variety in $\mathbb P^{n+2}$ by 
$F=0$ and $X_{n+2}^2=\Phi_2$, which is smooth and turns out to be the Galois closure variety. This proves Theorem \ref{101}.

We go to the proof of Proposition \ref{46}. 
Let $H$ be a general hyperplane containing the Galois line $\ell$ for $V$ in Theorem \ref{100}. 
Put $V_1=V \cap H$ and $D_1=D \cap H$. 
Since we are assuming $n \geq 2$, the $V_1$ is irreducible and nonsingular by Bertini's theorem. 
Thus, we have $\dim V_1=n-1$, $D_1^{n-1}=6$ and $V_1$ is also a smooth $(2,3)$-complete intersection. 
Note that $V_1 \sim D$ on $V$. Thus we have the exact sequence of sheaves 
\[
0 \longrightarrow \mathcal O_V \longrightarrow \mathcal O_V(V_1) \longrightarrow \mathcal O_{V_1}(D_1) \longrightarrow 0.  
\]
Taking cohomology, we get a long exact sequence 
\[
0 \longrightarrow {\mathop{\mathrm{H^0}}\nolimits}(V,\ {\mathcal O}_V) \longrightarrow {\mathop{\mathrm{H^0}}\nolimits}(V,\ {\mathcal O}_V(V_1)) \longrightarrow {\mathop{\mathrm{H^0}}\nolimits}(V_1,\ {\mathcal O}_{V_1}(D_1)) \]
$
\hspace{2.3cm} \longrightarrow {\mathop{\mathrm{H^1}}\nolimits}(V,\ {\mathcal O}_V) \longrightarrow \cdots.   
$  

\noindent Since $V$ is the complete intersection, we have ${\mathop{\mathrm{H^1}}\nolimits}(V,\ {\mathcal O}_V)=0$ (cf. \cite[III, Ex. 5.5]{har}). 
Then $V_1$ has the same properties as $V$ does, i.e., $\dim V_1=n-1$, $D_1^{n-1}=6$, 
 $\dim{\mathop{\mathrm{H^0}}\nolimits}(V_1,\ {\mathcal O}(D_1))=n+2$ and $\ell$ is a Galois line for $V_1$ and the Galois group is isomorphic to $D_6$.
Continuing the Descending Procedure, we get the sequence of Proposition \ref{46}. 

\medskip

There are a lot of problems concerning our theme, we pick up some of them.

\bigskip 

\noindent{\bf Problems.}

\begin{enumerate}
\item For each finite subgroup $G$ of $GL(2, k)$, does there exist a pair $(V, D)$ which defines the Galois embedding with the 
Galois group $G$ such that $D^n=|G|$, $\dim V=n$ and $\dim{\mathop{\mathrm{H^0}}\nolimits}(V,\ {\mathcal O}(D))=n+3$?
\item How many Galois subspaces do there exist for one Galois embedding? 
In case a smooth hypersurface $V$ in $\mathbb P^{n+1}$, there exist at most $n+2$. Further, it is $n+2$ if and only if 
$V$ is Fermat variety \cite{y3}.   
\item Does there exist a variety $V$ on which there exist two divisors $D_i$ ($i=1, 2$) such that they  
give Galois embeddings and $D_1^n \ne D_2^n$?  
\end{enumerate}

\bigskip 

For the detail, please visit our website 

http://mathweb.sc.niigata-u.ac.jp/~yosihara/openquestion.html

\bigskip

\bibliographystyle{amsplain}

\end{document}